\documentclass[12pt]{article}
\usepackage[centertags]{amsmath}
\usepackage{amsfonts}
\usepackage{amssymb}
\usepackage{latexsym}
\usepackage{amsthm}
\usepackage{newlfont}
\usepackage{graphicx}
\usepackage{listings}
\usepackage{booktabs}
\usepackage{abstract}
\usepackage{enumerate}
\RequirePackage{srcltx}
\date{}

\newlength{\defbaselineskip}
\newcommand{\setlinespacing}[1]%
           {\setlength{\baselineskip}{#1 \defbaselineskip}}

\newcommand{\N}{{\mathbb{N}}}

\newcommand{\actaqed}{\hfill $\actabox$}
{\medskip\noindent \textit{Proof of #1. }}%
{\actaqed \medskip}

\def\cA{{\mathcal A}}

\def\cF{{\mathcal F}}

\def\cH{{\mathcal H}}
\def \Tr{\mathcal T}

\def \cN{\mathcal N}

\def \cM{\mathcal M}
\def\R{{\mathbb R}}
\def\Z{\mathbb Z}

\def \T{\mathbb T}
\def\bP{\mathbb P}
\def\bE{\mathbb E}

\def \<{\langle}
\def\>{\rangle}

\def \ep{\epsilon}
\def \e{\varepsilon}

\def \ff{\varphi}

\def\bx{\mathbf x}

\def\bk{\mathbf k}

\def\bs{\mathbf s}

\newtheorem{Theorem}{Theorem}[section]
\newtheorem{Lemma}{Lemma}[section]

\newtheorem{Remark}{Remark}[section]

\numberwithin{equation}{section}

\newcommand{\be}{\begin{equation}}
\newcommand{\ee}{\end{equation}}

\begin{document}

\title{Sampling discretization of integral norms}
\author{ F. Dai, \, A. Prymak,\, A. Shadrin, \\ V. Temlyakov, \, and  \, S. Tikhonov  }

\newcommand{\Addresses}{{
  \bigskip
  \footnotesize

  F.~Dai, \textsc{ Department of Mathematical and Statistical Sciences\\
University of Alberta\\ Edmonton, Alberta T6G 2G1, Canada\\
E-mail:} \texttt{fdai@ualberta.ca }

  \medskip
 A.~Prymak, \textsc{ Department of Mathematics\\
University of Manitoba\\ Winnipeg, MB, R3T 2N2, Canada
  \\
E-mail:} \texttt{Andriy.Prymak@umanitoba.ca }

\medskip
A.~Shadrin, \textsc{Department of Mathematics and Theoretical Physics\\
University of Cambridge\\Wilberforce Road, Cambridge CB3 0WA, UK
\\
E-mail:} \texttt{a.shadrin@damtp.cam.ac.uk}

    \medskip
  V.N. Temlyakov, \textsc{University of South Carolina,\\ Steklov Institute of Mathematics,\\ and Lomonosov Moscow State University
  \\
E-mail:} \texttt{temlyak@math.sc.edu}

  \medskip

  S.~Tikhonov, \textsc{Centre de Recerca Matem\`{a}tica\\
Campus de Bellaterra, Edifici C
08193 Bellaterra (Barcelona), Spain;\\
ICREA, Pg. Llu\'{i}s Companys 23, 08010 Barcelona, Spain,\\
 and Universitat Aut\`{o}noma de Barcelona\\
E-mail:} \texttt{stikhonov@crm.cat}

}}

\maketitle
\begin{abstract}
{The paper is devoted to discretization of integral norms of functions from
a given finite dimensional subspace. Even though this problem is extremely important in applications, its systematic study has begun recently. In this paper we obtain a conditional theorem for all integral norms $L_q$, $1\le q<\infty$, which is an extension of known results for $q=1$.
To discretize  the integral norms successfully, we introduce a  new technique, which is
  a combination of probabilistic technique with results on the entropy numbers in the uniform norm.
  As an application of the general conditional theorem, we derive a new Marcinkiewicz type
   discretization for  the multivariate trigonometric polynomials with frequencies from the hyperbolic crosses.
  }
\end{abstract}

\section{Introduction}
\label{I}

 As it is clear from the title the two main concepts of the paper are {\it discretization} and {\it integral norms}.
Let $\Omega$ be a compact subset of $\R^d$ with the probability measure $\mu$. By $L_q$, $1\le q< \infty$, norm we understand
$$
\|f\|_q := \left(\int_\Omega |f|^qd\mu\right)^{1/q}.
$$
By discretization of the $L_q$ norm we understand a replacement of the measure $\mu$ by
a discrete measure $\mu_m$ with support on a set $\xi =\{\xi^\nu\}_{\nu=1}^m \subset \Omega$. This means that integration with respect to measure $\mu$ we replace by an appropriate cubature formula. Thus integration is replaced by evaluation of a function $f$ at a
finite set $\xi$ of points. This is why we call this way of discretization {\it sampling discretization}.

Discretization is
a very important step in making a continuous problem computationally feasible.
A prominent  example of a classical discretization problem is the problem of metric entropy (covering numbers, entropy numbers). The reader can find fundamental general results on metric entropy in \cite[Ch.15]{LGM},  \cite[Ch.3]{Tbook}, \cite[Ch.7]{VTbookMA}, \cite{Carl},
\cite{Schu} and in the recent papers \cite{VT156} and  \cite{HPV}.
Bounds for the entropy numbers of function classes are important by themselves and also have important connections to other fundamental problems (see, for instance, \cite[Ch.3]{Tbook} and \cite[Ch.6]{DTU}).

Another prominent example of a discretization problem is the problem of numerical integration.  Numerical integration in the mixed smoothness classes requires deep number theoretical results for constructing optimal (in the sense of order) cubature formulas (see, for instance, \cite[Ch.8]{DTU}). A typical approach to solving a continuous problem numerically -- the Galerkin method --
suggests to look for an approximate solution from a given finite dimensional subspace. A standard way to measure an error of approximation is an appropriate $L_q$ norm, $1\le q\le\infty$. Thus, the problem of   discretization of the $L_q$ norms of functions from a given finite dimensional subspace arises in a very natural way. The first results in this direction were obtained by Marcinkiewicz and
by Marcinkiewicz-Zygmund (see \cite{Z}) for discretization of the $L_q$ norms of the univariate trigonometric polynomials in 1930s. This is why we call discretization results of this kind the Marcinkiewicz-type theorems. There are different ways to discretize: use coefficients from an expansion with respect to a basis, more generally, use
linear functionals. We discuss here the way which uses function values at a fixed finite set of points.
We now proceed to the detailed presentation.

{\bf Marcinkiewicz problem.} Let $\Omega$ be a compact subset of $\R^d$ with the probability measure $\mu$. We say that a linear subspace $X_N$ (index $N$ here, usually, stands for the dimension of $X_N$) of $L_q(\Omega)$, $1\le q < \infty$, admits the Marcinkiewicz-type discretization theorem with parameters $m\in \N$ and $q$ if there exist a set $$\Big\{\xi^\nu \in \Omega: \nu=1,\dots,m\Big\}$$ and two positive constants $C_1(d,q)$ and $C_2(d,q)$ such that for any $f\in X_N$ we have
\be\label{I.1}
C_1(d,q)\|f\|_q^q \le \frac{1}{m} \sum_{\nu=1}^m |f(\xi^\nu)|^q \le C_2(d,q)\|f\|_q^q.
\ee
In the case $q=\infty$ we define $L_\infty$ as the space of continuous functions on $\Omega$  and ask for
\be\label{I.2}
C_1(d)\|f\|_\infty \le \max_{1\le\nu\le m} |f(\xi^\nu)| \le  \|f\|_\infty.
\ee
We will also use the following brief way to express the above properties: the $\cM(m,q)$ theorem holds for  a subspace $X_N$, written $X_N \in \cM(m,q)$.

The most complete results on sampling discretization are obtained in the case $q=2$.  The problem is basically solved in the case of subspaces of trigonometric polynomials. By $Q$ we denote a finite subset of $\Z^d$, and $|Q|$ stands for the number of elements in $Q$. Let
$$
\Tr(Q):= \left\{f: f=\sum_{\bk\in Q}c_\bk e^{i(\bk,\bx)},\  \  c_{\bk}\in\mathbb{C}\right\}.
$$
In \cite{VT158} it was shown how to derive the following result from the
 recent paper by  S.~Nitzan, A.~Olevskii, and A.~Ulanovskii~\cite{NOU}, which in turn is based on the paper of A.~Marcus, D.A.~Spielman, and N.~Srivastava~\cite{MSS}.

\begin{Theorem}\label{NOUth} There are three positive absolute constants $C_1$, $C_2$, and $C_3$ with the following properties: For any $d\in \N$ and any $Q\subset \Z^d$   there exists a set of  $m \le C_1|Q| $ points $\xi^j\in \T^d$, $j=1,\dots,m$ such that for any $f\in \Tr(Q)$
we have
$$
C_2\|f\|_2^2 \le \frac{1}{m}\sum_{j=1}^m |f(\xi^j)|^2 \le C_3\|f\|_2^2.
$$
\end{Theorem}

 Some results are obtained under an extra condition on the system $\{u_i(x)\}_{i=1}^N$.
 We will call it Condition E to be consistent with the prior work
 (see, e.g., \cite{DPTT}).
 \\
{\bf Condition E.} There exists a constant $t$ such that for all $x\in \Omega$
\be\label{ud5}
w(x):=\sum_{i=1}^N u_i(x)^2 \le Nt^2.
\ee
The reader can find the following result, which is a slight generalization of the Rudelson's \cite{Rud} celebrated result, in \cite{VT159}.

\begin{Theorem}\label{T5.4}  Let $\{u_i\}_{i=1}^N$ be a real  orthonormal system, satisfying condition~{\bf E}.
Then for every $\ep>0$ there exists a set $\{\xi^j\}_{j=1}^m \subset \Omega$ with
\[
m  \le C\frac{t^2}{\ep^2}N\log N
\]
such that for any $f=\sum_{i=1}^N c_iu_i$ we have
\[
(1-\ep)\|f\|_2^2 \le \frac{1}{m} \sum_{j=1}^m f(\xi^j)^2 \le (1+\ep)\|f\|_2^2.
\]
\end{Theorem}

Rather complete results on sampling discretization of the $L_1$ norm are obtained in \cite{VT159}. One of the main goals of this paper is to treat the case $1<q<\infty$.   In Section \ref{CondT} we present a generalization of discretization result from \cite{VT159}, which treats the case $q=1$, to the case of $1<q<\infty$. We give a detailed proof in Section \ref{CondT}.
We note that the case $q=2$ is much better developed both in the case of trigonometric polynomials and in the general case (see \cite{VT158}, \cite{VT159}, \cite{DPTT}, and the above discussion).
We prove in Section \ref{CondT} the following conditional result. Denote
$$
X^q_N := \{f\in X_N:\, \|f\|_q \le 1\}.
$$
For the definition of the entropy numbers $\e_k$ see Section \ref{CondT}.
\begin{Theorem}\label{T2.1} Let $1\le q<\infty$. Suppose that a subspace $X_N$ satisfies the condition
\be\label{I6}
\e_k(X^q_N,L_\infty) \le  B (N/k)^{1/q}, \quad 1\leq k\le N,
\ee
where $B\ge 1$.
Then for large enough constant $C(q)$ there exists a set of
$$
m \le C(q)NB^{q}(\log_2(2BN))^2
$$
 points $\xi^j\in \Omega$, $j=1,\dots,m$,   such that for any $f\in X_N$
we have
$$
\frac{1}{2}\|f\|_q^q \le \frac{1}{m}\sum_{j=1}^m |f(\xi^j)|^q \le \frac{3}{2}\|f\|_q^q.
$$
\end{Theorem}

\begin{Remark}\label{rem:nik}
 We remark that 
 the  assumption \eqref{I6} for $k=1$ implies the following Nikol'skii type inequality for $X_N$:
	\[
	\|f\|_\infty \le 4BN^{1/q}\|f\|_q\   \ \text{for any $f\in X_N$.}
	\]
This  is shown in the beginning of the proof of Theorem \ref{T2.1}.
\end{Remark}

\begin{Remark}
In the proof of Theorem \ref{T2.1} we use the technique, which we call {\textnormal{ sandwiching}}. This means that we find a mapping of the $X_N^q$ into a finite set $\{h(f,\bx)\}$ of
piecewise constant functions with a property $C_1(a) h(f,\bx) \le |f(\bx)| \le C_2(a) h(f,\bx)$, for a large set of $\bx$. The idea of {\it sandwiching} is related to the idea of {\it entropy with bracketing}, which is widely used in the empirical process theory (see \cite{vdG}). We realize the sandwiching idea in the form very close to the one used in the paper by E.S. Belinsky \cite{Bel}. In the case $q=1$ Theorem \ref{T2.1} was proved in \cite{VT159} with the help of a different technique -- the chaining technique.

\end{Remark}
Applications of Theorem~\ref{T2.1} and further discussions are given in Section~3.

This paper can be considered  a natural continuation of  the recent papers~\cite{VT158}, \cite{VT159} and~\cite{DPTT}.

\section{Conditional theorem for discretization in $L_q$}
\label{CondT}

We begin with the definition of the entropy numbers.
  Let $(X,\|\cdot\|)$ be a Banach space and let $B_X$ denote the unit ball of $X$ with the center at the origin. Denote by $B_X(f,r)$ a ball with center $f$ and radius $r$: $\{g\in X:\|f-g\|\le r\}$. For a compact set $A$ and a positive number $\e$ we define the covering number $N_\e(A)$
 as follows
$$
N_\e(A) := N_\e(A,X)
:=\min \Biggl\{n : \exists f_1,\dots,f_n, f_j\in A , A\subseteq \bigcup_{j=1}^n B_X(f_j,\e)\Biggr\}.
$$
The corresponding minimal $\e$-net is denoted by $\cN_\e(A,X)$. Thus, $N_\e(A,X)= |\cN_\e(A,X)|$.
It is convenient to consider along with the $\varepsilon$-entropy $$\cH_\e(A,X):= \log_2 N_\e(A,X)$$ the entropy numbers $\e_k(A,X)$ of the set $A$ in $X$:
\begin{align*}
\e_k(A,X)  :&=\inf \big\{\e>0: \cH_\e (A; X)\leq k\big\},\   \  k=1,2,\cdots.
\end{align*}
In our definition of $N_\e(A)$ and $\e_k(A,X)$ we require $f_j\in A$. In a standard definition of $N_\e(A)$ and $\e_k(A,X)$ this restriction is not imposed.
However, it is well known (see \cite[p.208]{Tbook}) that these characteristics may differ at most by a factor $2$.

{\bf Proof of Theorem \ref{T2.1}.} We begin with the proof of Remark~\ref{rem:nik}.
Without loss of generality assume that $f\in X_N$ and $\|f\|_q=1$. By our assumption $\e_1:= \e_1(X_N^q,L_\infty) \le BN^{1/q}$. Since $X_N^q$ is compact, it means that there exist two elements $f_1$ and $f_2$ in $X_N^q$ such that $X_N^q\subset B_{L_\infty}(f_1,\e_1)\cup B_{L_\infty}(f_2,\e_1)$. This implies that the zero element belongs to one of these balls, say, $B_{L_\infty}(f_1,\e_1)$  and, therefore, $\|f_1\|_\infty\le \e_1$.
Next, $(f_1+f_2)/2$ belongs to one of those balls and, therefore, $\|f_1-f_2\|_\infty\le 2\e_1$. Thus, 
for any $f\in X_N^q$ we have
\be\label{2.4}
\|f\|_\infty \le 4\e_1 \le 4BN^{1/q}.
\ee

\begin{Lemma}\label{lem:entropy bound for all eps}
	The condition~\eqref{I6} implies
	\be\label{eqn:entropy bound}
	\cH_\e(X^q_N,L_\infty)\le 1+N
	\begin{cases}
		(B/\e)^q & \text{if } \e\ge B,\\
		\log_2(6B/\e) & \text{if } 0<\e<B.
	\end{cases}
	\ee
\end{Lemma}
\begin{proof}
We begin by pointing out that the assumption \eqref{I6} for $k=N$ implies the inequality
\begin{equation}\label{eqn:k more than N}
\e_k(X_N^q, L_\infty)\le 6B2^{-k/N}  \   \  \text{ for $k>N$}.
\end{equation}
This  follows directly  from the facts that for each Banach  space $X$ (see \cite[(7.1.6), p. 323]{VTbookMA}),
$$
\e_k(A,X) \le \e_N(A,X)\e_{k-N}(B_X,X), \  \ k>N,
$$
and for each $N$-dimensional space $X$ (see \cite[Corollary 7.2.2, p. 324]{VTbookMA}),
$$
\varepsilon_m(B_X,X) \le 3(2^{-m/N}).
$$

Now we will establish~\eqref{eqn:entropy bound}. If $\e>BN^{\frac 1q}$, then $\cH_{\e} (X_N^q, L_\infty)\leq 1$ by~(\ref{I6}). If $B<\e\leq BN^{\frac 1q}$, then taking an integer  $k\in [2,N]$ satisfying
$$
B(N/k)^{1/q} < \e \leq  B(N/(k-1))^{1/q},
$$
we get from~(\ref{I6})
\[
\cH_\e(X_N^q,L_\infty) \le k \le N(B/\e)^q +1.
\]		
Similarly, using~\eqref{eqn:k more than N}, for $\e\le 3B$ we obtain
\begin{equation*}
\cH_\e(X_N^q,L_\infty) \le  N\log_2(6B/\e) +1.
\end{equation*}
\end{proof}

Using the inequality $\ln(t)\le C(q) t^q$ for $t>1$ and~\eqref{eqn:entropy bound}, we have that for every ${\e'}>0$ and $q\in[1,\infty)$ there exists $C({\e'},q)>0$ such that
\begin{equation}\label{eqn:entropy bound eps0}
\cH_\e(X^q_N,L_\infty)\le C({\e'},q) (B/\e)^q \quad\text{for any }\e\ge {\e'}.
\end{equation}

{\bf Sandwiching construction.} Let $a\in(0,1/2]$ be a small number, which will be chosen later.
Denote
$$
\cA_j := \cN_{a(1+a)^j}(X_N^q,L_\infty), \qquad j\in\Z.
$$

 Take a number $j_0\in \Z$, which will be specified later, and for $j\in \Z$, $j\ge j_0$,  define a mapping $A_j$ that associates with a function $f\in X_N^q$ a function $A_j(f) \in {\cA}_j$ closest to $f$ in the $L_\infty$ norm. Then, clearly,
$$
\|f-A_j(f)\|_\infty \le a(1+a)^j.
$$
We now define a mapping of  $X_N^q$ to a finite set of piecewise constant functions.
For $f\in X_N^q$ and $j\in \Z$ denote
$$
U_j(f) := \{\bx : |A_j(f)(\bx)| \ge (1+a)^{j-1}\},
$$
$$
D_j(f) := U_j(f) \setminus \cup_{k>j} U_k(f),\qquad D_{j_0}(f) := \Omega \setminus \cup_{k>j_0} U_k(f).
$$
Define
$$
h(f,\bx) := \sum_{j>j_0} (1+a)^j \chi_{D_j(f)}(\bx),
$$
where $\chi_E(\bx)$ is a characteristic function of a set $E$.

We need some properties of the above $h$-mapping. We will sandwich $|f|$
by functions $C_1(a)h(f)$ and $C_2(a)h(f)$. We do it on each $D_j(f)$, $j>j_0$. By the definition of $D_j(f)$, the  condition $\bx\in D_j(f)$ implies that $\bx\in U_j(f)$ and $\bx \notin U_{j+1}(f)$. From the definition of the $U_j(f)$ we obtain for $\bx \in U_j(f)$
\[
|f(\bx)| \ge |A_j(f)| - a(1+a)^j \ge (1+a)^{j-1} - a(1+a)^j = (1+a)^j C_1(a),
\]
where
$$
C_1(a):= \frac{1-a(1+a)}{1+a}.
$$
From the definition of the $U_{j+1}(f)$ we obtain for $\bx \notin U_{j+1}(f)$
\be\label{2.8}
|f(\bx)| \le |A_{j+1}(f)| + a(1+a)^{j+1} \le (1+a)^{j} + a(1+a)^{j+1} = (1+a)^j C_2(a),
\ee
where
$$
C_2(a):=1+a(1+a).
$$
Therefore, for all $\bx \in \Omega \setminus D_{j_0}(f)$ we have
\be\label{2.9}
C_1(a)h(f,\bx) \le |f(\bx)| \le C_2(a)h(f,\bx).
\ee
It is clear that
\be\label{2.10}
\lim_{a\to 0} C_1(a) = \lim_{a\to 0} C_2(a) =1.
\ee
In the same way as we obtained the bound (\ref{2.8}) we derive for $\bx\in D_{j_0}(f)$
\begin{align}
|f(\bx)|  & \le |A_{j_0+1}(f)| + a(1+a)^{j_0+1} \nonumber \\
\label{2.11}
 & \le (1+a)^{j_0} + a(1+a)^{j_0+1} = (1+a)^{j_0} C_2(a).
\end{align}

We now show that a good discretization formula for the  functions $h(f)$ is also good for the  functions $f$. For a point set $\xi=\{\xi^\nu\}_{\nu=1}^m$ and a function $f$, denote
$$
S(f,\xi) := (f(\xi^1),\dots,f(\xi^m)) \in \R^m,\quad \|S(f,\xi)\|_q^q:= \frac{1}{m}\sum_{\nu=1}^m |f(\xi^\nu)|^q.
$$

\begin{Lemma}\label{BL1} Let $q\in [1,\infty)$. Assume that a point set $\xi$ is such that
for a function $h(f,\bx)$ we have the inequalities
\be\label{2.12}
\|h(f)\|_q^q-\delta \le \|S(h(f),\xi)\|_q^q \le  \|h(f)\|_q^q +\delta
\ee
for some constant $\delta>0$.
Then
\begin{multline*}
C_1(a)^q\Bigl (C_2(a)^{-q}\bigl(\|f\|_q^q - C_2(a)^q (1+a)^{qj_0}\bigr)-\delta\Bigr)\\ \le \|S(f,\xi)\|_q^q
\le C_2(a)^q (1+a)^{qj_0} + C_2(a)^q \Bigl(C_1(a)^{-q} \|f\|_q^q+\delta\Bigr).
\end{multline*}

\end{Lemma}
\begin{proof} First, we take care of the set $D_{j_0}(f)$. By (\ref{2.11}) we have
$$
\int_{D_{j_0}(f)} |f(\bx)|^q d\mu \le C_2(a)^q (1+a)^{qj_0}
$$
and
$$
\frac{1}{m} \sum_{\nu: \xi^\nu\in D_{j_0}(f)}|f(\xi^\nu)|^q \le C_2(a)^q(1+a)^{qj_0}.
$$
Then, on the one hand, by (\ref{2.9}) we have
$$
\|S(f,\xi)\|_q^q \le C_2(a)^q (1+a)^{qj_0} + C_2(a)^q \|S(h(f),\xi)\|_q^q.
$$
Using (\ref{2.12}) and (\ref{2.9}), we continue
\begin{eqnarray*}
\|S(f,\xi)\|_q^q &\le& C_2(a)^q (1+a)^{qj_0} + C_2(a)^q (\|h(f)\|_q^q+\delta)
\\
 &\le& C_2(a)^q (1+a)^{qj_0} + C_2(a)^q (C_1(a)^{-q} \|f\|_q^q+\delta).
\end{eqnarray*}
On the other hand by (\ref{2.9}) and (\ref{2.12}) we have
\begin{eqnarray*}
\|S(f,\xi)\|_q^q &\ge& C_1(a)^q \|S(h(f),\xi)\|_q^q  \ge C_1(a)^q(\|h(f)\|_q^q -\delta)
\\
&\ge& C_1(a)^q (C_2(a)^{-q}\int_{\Omega\setminus D_{j_0}(f)} |f(\bx)|^qd\mu -\delta)
\\
&\ge& C_1(a)^q (C_2(a)^{-q}(\|f\|_q^q - C_2(a)^q (1+a)^{qj_0})-\delta).
\end{eqnarray*}

\end{proof}

\begin{Remark}\label{BR1} Under the assumption $\|f\|_q^q =1/2$ and $\delta =1/8$, using (\ref{2.10}), we can choose $j_0=j_0(a)$ and $a=a(q)$ such that Lemma \ref{BL1} gives
$$
\frac{1}{2}\|f\|_q^q \le \|S(f,\xi)\|_q^q \le \frac{3}{2}\|f\|_q^q
$$
and, in addition, $C_1(a)^{-q} \le 2$.
\end{Remark}

{\bf Existence of good $\xi$.} Let $q\in [1,\infty)$ and $a$, $j_0$, be from Remark \ref{BR1}. Although $a$ and $j_0$ depend on $q$ from now on, we will keep indicating the dependence of constants on $a$ and $j_0$ for clarity.
For $j>j_0$ consider the following sets of piecewise constant functions
$$
\cF_j^q := \left\{ (1+a)^{qj}\chi_{D_j(f)},\, f\in X_N,\, \|f\|_q^q =1/2\right\}.
$$
Our argument is based on   \cite[Lemma~2.1]{BLM}.

\begin{Lemma}\label{AL1} Let $\{g_\nu\}_{\nu=1}^m$ be independent random variables with $\bE g_\nu=0$, $\nu=1,\dots,m$, which satisfy
$$
\|g_\nu\|_1\le 2,\qquad \|g_\nu\|_\infty \le M,\qquad \nu=1,\dots,m.
$$
Then for any $\eta \in (0,1)$ we have the following bound on the probability
$$
\bP\left\{\left|\sum_{\nu=1}^m g_\nu\right|\ge m\eta\right\} < 2\exp\left(-\frac{m\eta^2}{8M}\right).
$$
\end{Lemma}

It is easy to see that
 Lemma \ref{AL1} implies the following result.

\begin{Lemma}\label{BL2} Let $\{\cF_j\}_{j\in G}$ be a collection of finite sets of
functions from $L_1(\Omega,\mu)$. Assume that for each $j\in G$ and all $f\in \cF_j$ we have
$$
\|f\|_1 \le 1,\quad \|f\|_\infty \le M_j.
$$
Suppose that positive numbers $\eta_j$ and a natural number $m$ satisfy the condition
$$
2\sum_{j\in G} |\cF_j| \exp \left(-\frac{m\eta_j^2}{8M_j}\right) <1.
$$
Then there exists a set $\xi=\{\xi^\nu\}_{\nu=1}^m \subset \Omega$ such that for each $j\in G$ and for all $f\in \cF_j$ we have
$$
\left|\|f\|_1 - \frac{1}{m}\sum_{\nu=1}^m |f(\xi^\nu)|\right| \le \eta_j.
$$
\end{Lemma}

We apply Lemma \ref{BL2} for a collection of the above sets $\cF_j^q$. First of all, it is clear
from (\ref{2.9}) and the choice of parameters $a$ and $j_0$ (see Remark \ref{BR1}) that we have
$$
\|(1+a)^{qj}\chi_{D_j(f)}\|_1 \le \|h(f)\|_q^q \le C_1(a)^{-q}\|f\|_q^q \le 1.
$$
Secondly, obviously,
$$
\|(1+a)^{qj}\chi_{D_j(f)}\|_\infty \le (1+a)^{qj} =: M_j.
$$
It is clear from (\ref{2.4}) that we only need to consider those $j$, which satisfy the condition
$(1+a)^{j-1}\le 4BN^{1/q}$. Indeed, if this condition is not satisfied, then by $A_j(f)\in X^q_N$ we have $U_j(f)= \varnothing$.

Let $J\in\Z$ be the one satisfying
$$
(1+a)^{J-1}\le 4BN^{1/q} < (1+a)^J.
$$
Then
\[
J \le 1+ \log(4BN^{1/q})/\log(1+a).
\]
Denote $G:= (j_0,J]\cap \mathbb{Z}$. Then $|G| = J - j_0 \le J+|j_0|$. Set $\eta_j = \frac{1}{8|G|}$. Then
\be\label{2.14}
\sum_{j\in G} \eta_j \le 1/8.
\ee
We now estimate cardinalities $|\cF_j^q|$ for $j\in G$. From the definition of $\cF_j^q$ and the construction of the sets $D_j(f)$, we conclude that
$$
|\cF_j^q| \le |\cA_j|\times\cdots\times|\cA_J|.
$$
Therefore,
$$
\ln |\cF_j^q| \le \sum_{k=j}^J \ln |\cA_k|.
$$
By~\eqref{eqn:entropy bound eps0} with ${\e'}:=a(1+a)^{j_0}$ we obtain for $k\in G$
$$
\ln |{\cA_k}| \le \cH_{a(1+a)^k}(X_N^q,L_\infty) \le C_3(a,j_0,q)NB^q(1+a)^{-qk}.
$$
Therefore,
$$
\ln |\cF_j^q| \le C_4(a,j_0,q) NB^q (1+a)^{-qj}.
$$

 We now choose $C(a,j_0,q)$ large enough to guarantee that for any $m\ge 2C(a,j_0,q)NB^q|G|^2$ we have
 $$
 C_4(a,j_0,q) NB^q (1+a)^{-qj}-\frac{m\eta_j^2}{8M_j} \le - C(a,j_0,q) NB^q (1+a)^{-qj}
 $$
and
$$
2\sum_{j\in G} \exp\left(- C(a,j_0,q) NB^q (1+a)^{-qj}\right) <1,
$$
where  we have used that $\exp(-t)<\frac1t$ for $t>0$ and (2.12) in the last inequality.

Then Lemma \ref{BL2} provides the existence of $\xi = \{\xi^\nu\}_{\nu=1}^m$ such that for each
$j\in G$ and all $\ff_j \in \cF_j^q$ we have
\be\label{2.15}
\left| \|\ff_j\|_1 - \frac{1}{m}\sum_{\nu=1}^m |\ff_j(\xi^\nu)|\right| \le \eta_j.
\ee

Let $f\in X_N$ be such that $\|f\|_q^q =1/2$. We now prove (\ref{2.12}) for the above chosen
$\xi$ with $\delta=1/8$. Specify $\ff_j = (1+a)^{qj}\chi_{D_j(f)}$. Then, taking into account the fact that the sets $\{D_j(f)\}_{j\in G}$ are disjoint, we obtain
\be\label{2.16}
\|h(f)\|_q^q = \sum_{j\in G} \|\ff_j\|_1,\qquad \|S(h(f),\xi)\|_q^q =\sum_{j\in G} \frac{1}{m}\sum_{\nu=1}^m |\ff_j(\xi^\nu)|.
\ee
Inequalities (\ref{2.12}) follow from (\ref{2.16}), (\ref{2.15}), and (\ref{2.14}). This completes the proof of Theorem \ref{T2.1}.

\begin{Remark}
	The same technique of the proof gives a slightly more general statement. Namely, suppose $q,\alpha\in[1,\infty)$ and that a subspace $X_N$ satisfies the condition
	\[
	\e_k(X^q_N,L_\infty) \le  B (N/k)^{\alpha/q}, \quad 1\leq k\le N,
	\]
	where $B\ge 1$.
	Then for large enough constant $C(\alpha,q)$ there exists a set of
	$$
	m \le C(\alpha,q)N^\alpha B^{q}(\log_2(2BN))^2
	$$
	points $\xi^j\in \Omega$, $j=1,\dots,m$,   such that for any $f\in X_N$
	we have
	$$
	\frac{1}{2}\|f\|_q^q \le \frac{1}{m}\sum_{j=1}^m |f(\xi^j)|^q \le \frac{3}{2}\|f\|_q^q.
	$$
\end{Remark}

\section{Discussion}
\label{D}

In this section we discuss the Marcinkiewicz-type discretization theorems for the hyperbolic cross trigonometric polynomials. For $\bs\in\Z^d_+$
define
$$
\rho (\bs) := \big\{\bk \in \Z^d : [2^{s_j-1}] \le |k_j| < 2^{s_j}, \quad j=1,\dots,d\big\},
$$
where $[x]$ denotes the integer part of $x$. We define the step hyperbolic cross
$Q_n$ as follows
$$
Q_n := \cup_{\bs:\|\bs\|_1\le n} \rho(\bs)
$$
and the corresponding set of the hyperbolic cross polynomials as
$$
\Tr(Q_n) := \big\{f: f=\sum_{\bk\in Q_n} c_\bk e^{i(\bk,\bx)}\big\}.
$$
It is worth mentioning that
$|Q_n|\asymp 2^n n^{d-1}$.
The following theorem was proved in \cite{VT159}.
\begin{Theorem}\label{T3.2} For any $d\in \N$ and $n\in \N$ for large enough absolute constant $C_1(d)$ there exists a set of  $m \le C_1(d)|Q_n|n^{7/2}$ points $\xi^j\in \T^d$, $j=1,\dots,m$ such that for any $f\in \Tr(Q_n)$
we have
$$
C_2\|f\|_1 \le \frac{1}{m}\sum_{j=1}^m |f(\xi^j)| \le C_3\|f\|_1.
$$
\end{Theorem}
The proof of Theorem \ref{T3.2} from \cite{VT159} is based on conditional Theorem \ref{T2.1} with $q=1$ and the bounds for the entropy numbers. We note that the problem of estimating the entropy numbers in the $L_\infty$ norm is a nontrivial problem by itself. We demonstrate this on the example of the trigonometric polynomials. It is proved in \cite{VT156} that in the case $d=2$
we have
\be\label{6.1}
\e_k(\Tr( Q_n)_1,L_\infty)\ll  n^{1/2} \left\{\begin{array}{ll} (| Q_n|/k) \log (4| Q_n|/k), &\quad k\le 2| Q_n|,\\
2^{-k/(2| Q_n|)},&\quad k\ge 2| Q_n|.\end{array} \right.
\ee
The proof of estimate (\ref{6.1}) is based on an analog of the Small Ball Inequality for the trigonometric system proved for the wavelet type system (see \cite{VT156}). This proof uses the two-dimensional specific features of the problem and we do not know how to extend this proof to the case $d>2$. Estimate (\ref{6.1}) is used in the proof of the right order  upper bounds  for
the classes of mixed smoothness (see \cite{VT156}). This means that (\ref{6.1}) cannot be substantially improved. However, in application to the Marcinkiewich-type theorem we use the trivial inequality $\log (4| Q_n|/k) \ll n$ and the following corollary of (\ref{6.1})
\be\label{6.2}
\e_k(\Tr( Q_n)_1,L_\infty)\ll  n^{3/2} \left\{\begin{array}{ll} | Q_n|/k , &\quad k\le 2| Q_n|,\\
2^{-k/(2| Q_n|)},&\quad k\ge 2| Q_n|.\end{array} \right.
\ee
It turns out that the following upper bound from \cite{Bel}, which applies for all $d$, gives a better result for the Marcinkiewich-type theorem: for $1\le q\le 2$ we have
\be\label{6.3}
\e_k(\Tr( Q_n)_q,L_\infty)\ll  n^{1/q} \left\{\begin{array}{ll} | Q_n|/k , &\quad k\le 2| Q_n|,\\
2^{-k/(2| Q_n|)},&\quad k\ge 2| Q_n|.\end{array} \right.
\ee
A combination of (\ref{6.3}) and Theorem \ref{T2.1} gives the following Marcinkiewich-type theorem for the hyperbolic cross trigonometric polynomials.

\begin{Theorem}\label{T3.3} Let $1\le q\le 2$. There is a number $C(d,q)$ such that for any $n\in \N$ there   exists a set of $m \le C(d,q)|Q_n| n^{3}$ points $\xi^j\in \T^d$, $j=1,\dots,m$ such that for any $f\in \Tr(Q_n)$
we have
$$
\frac{1}{2}\|f\|_q^q \le \frac{1}{m}\sum_{j=1}^m |f(\xi^j)|^q \le \frac{3}{2}\|f\|_q^q.
$$
\end{Theorem}

We note that Belinsky \cite{Bel} obtained an analog of Theorem \ref{T3.3} with a somewhat weaker bound $m\le C(d,q)|Q_n| n^{4}$. Also, the corresponding proof in \cite{Bel} contains some inaccuracies.

 We point out that the situation with the discretization theorems in the $L_\infty$ case is fundamentally different. A  nontrivial surprising negative result was proved for the $L_\infty$ case (see \cite{KT3}, \cite{KT4}, and \cite{KaTe03}). The authors proved that the necessary condition for
$\Tr(Q_n)\in\cM(m,\infty)$ is $m\gg |Q_n|^{1+c}$ with absolute constant $c>0$.
We refer the reader to \cite{DPTT} for further results on discretization in the $L_\infty$ norm.


  {\bf Acknowledgment.} The work was supported by the Russian Federation Government Grant N{\textsuperscript{\underline{o}}}14.W03.31.0031. The paper contains results obtained in frames of the program
 ``Center for the storage and analysis of big data'', supported by the Ministry of Science and High Education of Russian Federation (contract 11.12.2018N{\textsuperscript{\underline{o}}}13/1251/2018 between the Lomonosov Moscow State University and the Fond of support of the National technological initiative projects).
The first named author's research was partially supported by NSERC of Canada Discovery Grant RGPIN 04702-15.
The second named author's research was partially supported by NSERC of Canada Discovery Grant RGPIN 04863-15.
The fourth named author's research was supported by the Russian Federation Government Grant No. 14.W03.31.0031.
The fifth  named
author's research was partially supported by
 MTM 2017-87409-P,  2017 SGR 358, and
 the CERCA Programme of the Generalitat de Catalunya.

\Addresses




\end{document}